\documentclass[12pt]{amsart}
\usepackage{fullpage,url,amssymb, enumerate,colonequals}
\usepackage{mathtools}
\usepackage{mathrsfs} 
\usepackage{caption}
\usepackage{MnSymbol}
\usepackage{extarrows}
\usepackage{lscape}
\usepackage{xfrac}
\usepackage[all,cmtip]{xy}
\usepackage{tablefootnote}
\usepackage[OT2,T1]{fontenc}
\usepackage{color}
\usepackage{marginnote}

\usepackage{xr-hyper}

\usepackage[
        colorlinks, citecolor=darkgreen,
        backref,
        pdfauthor={}, 
]{hyperref}
\usepackage{cleveref}
\usepackage{comment}
\usepackage{todonotes}

\usepackage[pdftex]{epsfig}

\numberwithin{equation}{section}

\newtheorem{lemma}[equation]{Lemma}

\newtheorem{prop}[equation]{Proposition}

\newtheorem{claim*}{Claim}
\newtheorem{theorem}[equation]{Theorem}

\newtheorem{ques}[equation]{Question}
\newtheorem{example}[equation]{Example}
\newtheorem{remark}[equation]{Remark}

\definecolor{darkgreen}{rgb}{0,0.5,0}
\definecolor{rem}{rgb}{0.8,0,0}
\definecolor{new}{rgb}{0.3,0.1,0.9}
\definecolor{reply}{rgb}{0,0,0.8}
\definecolor{gray}{gray}{0.7}

\renewcommand{\det}{\text{det}}

\newcommand{\F}{\mathbb{F}}

\newcommand{\Q}{\mathbb{Q}}

\newcommand{\Z}{\mathbb{Z}}
\newcommand{\rhobar}{{\overline{\rho}}}
\newcommand{\eps}{\varepsilon}

\newcommand{\Qbar}{{\overline{\Q}}}
\newcommand{\Gal}{\text{Gal}}
\newcommand{\Frob}{\text{Frob}}

\newcommand{\GL}{\text{GL}}

\DeclareMathOperator{\tr}{tr}

\newcommand{\oo}{\mathcal{O}}

\newcommand{\gl}[2]{\text{GL}_{#1}(#2)}
\newcommand{\rsr}[1]{\overline{#1}^{ss}}
\newcommand{\rr}[1]{\overline{#1}}

\newcommand{\LMFDBE}[1]{\href{https://www.lmfdb.org/EllipticCurve/Q/#1}{\textsf{#1}}}

\newcommand{\LMFDBN}[1]{\href{https://www.lmfdb.org/ModularForm/GL2/Q/holomorphic/#1}{\textsf{#1}}}

\title{Congruences and ramified primes in fields of coefficients of newforms}
\author{Nuno Freitas and Filip Gawron}
\date{\today}

\begin{document}
\maketitle
\begin{abstract}
    We investigate the splitting behavior of $\ell$ in the coefficient field of a newform~$f$ of level $N$, under the assumption that $f$ is congruent modulo a prime above $\ell$ to another newform $g$ whose level divides $N/p^2$ for some prime $p|N$. In particular, we show that the maximal real subfield of the $\ell$-th cyclotomic field, $\Q(\zeta_\ell + \zeta_\ell^{-1})$, is contained in the coefficient field of $f$. We conclude by presenting explicit examples that illustrate these results.
\end{abstract}
\section{Introduction}
Let $f\in S_k(\Gamma_0(N),\varepsilon)$ be a newform of weight $k \geq 2$, level $N$ and nebentypus $\varepsilon$. It is a standard fact that~$f$ admits a Fourier expansion
$f(q)=\sum_{n=1}^\infty a_n(f)q^n$
and a field of coefficients
$\Q_f:=\Q(\{a_n(f)\}_{n\geq 1})$
which is a finite extension of~$\Q$.

Although modular forms have been extensively studied, there are still many open questions about the field of coefficients $\Q_f$. Here we are particularly interested in the following.
\begin{ques}
    What can be said about the splitting behaviour of primes in $\Q_f$ in terms of the level $N$ and the nebentypus~$\varepsilon$?
\end{ques}

It is well known that $\Q_f$ always contains all the values taken by~$\varepsilon$ (e.g. \cite[Corollary 3.1]{Ribneb}). An example of a result in terms of $N$ is due to Brumer (see \cite[Theorem 5.5]{Bru}) who proved that if
$\ell^\alpha$ divides $N$, where $\ell$ is a prime, then $\Q_f$ contains $\Q(\zeta_{\ell^\beta}+\zeta_{\ell^\beta}^{-1})$ for some $\beta=\beta(\alpha)$, an increasing function of $\alpha$.
A result of similar flavor is \cite[Proposition 4.7]{DieWie}, where the authors proved that if $f$ has level $Mp^2$, with $p\nmid M$, and possesses certain inertial type at $p$ (depending on an integer~$n$), then $\Q_f$ contains $\Q(\zeta_n+\zeta_n^{-1}).$

Let $G_\Q := \Gal(\Qbar /\Q)$. A major tool for studying modular forms is its associated Galois representations. More precisely, Deligne~\cite{Del} showed that, to every newform $f$ as above, and any prime $\lambda \mid \ell$ in~$\Q_f$, there is a Galois representation
$\rho_{f,\lambda}:G_\Q\to \gl{2}{\oo_{f,\lambda}}$,
where $\oo_{f,\lambda}$ is the ring of integers of the completion of $\Q_f$ at $\lambda$. Furthermore, this representation is unramified outside~$\ell N$ and has a prime-to-$\ell$ Artin conductor equal to the prime-to-$\ell$ part of~$N$. For any prime $p \neq \ell$, we will refer to the restriction $\rho_{f,\lambda}|_{I_p}$, where $I_p \subset G_\Q$ is an inertia subgroup at~$p$, as the {\it inertial type} of $f$ at~$p$.

Before stating our main result, we need to introduce some notation.
Let $\rr{\rho_{f,\lambda}}:G_\Q\to\gl{2}{\overline{\mathbb{F}_\ell}}$
be the modulo~$\lambda$ reduction of $\rho_{f,\lambda}$ and
denote by $N'$ the prime-to-$\ell$ part of the Artin conductor of
its semisimplification $\rsr{\rho_{f,\lambda}}$.
We will prove that, whenever $v_p(N)> 1 $ and $\overline{\rho_{f,\lambda}}^{ss}$ `loses' ramification at a prime $p\neq \ell$, i.e., $v_p(N')<v_p(N)$, then $\ell$ ramifies in~$\Q_f$.
More precisely, we will prove the following theorem.
\begin{theorem}\label{maintheorem}
Let $f\in S_k(\Gamma_0(N),\varepsilon)$ be a newform of weight $k\geq 2$
and field of coefficients~$\Q_f$. Let $p$ be a prime such that $p^2 \mid N$.
Let $\ell \neq p$ be an odd prime, $\lambda$ a prime of $\Q_f$ dividing~$\ell$, and
$N'$ the prime-to-$\ell$ part of the Artin conductor of $\rsr{\rho_{f,\lambda}}$.
Assume that
$v_p(N')<v_p(N).$
\begin{enumerate}[(a)]
    \item If $v_p(N)=2$, then
     $\Q(\zeta_\ell+\zeta_\ell^{-1})\subset \Q_f.$
     In particular, for any prime $\lambda'$ of $\Q_f$ over $\ell$ we have
    $$\frac{\ell-1}{2}\mid e(\lambda'|\ell),$$
    where $e(\lambda'|\ell)$ is the ramification degree of $\lambda'$ over $\ell$ in $\Q_f$. Moreover, if $v_p(N')=0$ and $\overline{\rho_{f,\lambda}}$ is irreducible then
    $$e(\lambda|\ell)\geq \ell-1.$$
    \item If $v_p(N)>2$, then $\Q(\zeta_\ell)\subset \Q_f$.
\end{enumerate}
\end{theorem}
This theorem generalizes \cite[Lemmas 15 and 16]{DiePaTsa} as it allows for nontrivial nebentypus, but it also provides stronger conclusions. Indeed, 
case~(b) gives the containment of the full cyclotomic field in $\Q_f$ and, moreover, when $v_p(N')=0$, our result establishes the existence of a prime above $\ell$ with ramification degree bigger than what can be deduced from the field containment alone.

The ideas in the proof of this theorem are as follows. The statements about the fields contained in~$\Q_f$ follow from careful analysis of the inertial type of $f$ at $p$ and their reductions. Under the assumption $v_p(N')<v_p(N)$, the classification of the possible inertial types is due to Carayol (see Proposition \ref{Carayolcharacterization}), and we use it as the starting point. The argument is similar to the one in \cite{DiePaTsa}.
The last statement of part $(a)$ is the only non-symmetric statement with respect to the primes $\lambda$ of $\Q_f$ above $\ell$. Its proof goes by imitating the trick used in \cite[\textsection 2, Cor. 2]{Se--Ta}, which relies on the fact that when the ramification degree of $\Q_f/\Q$ is small enough, the kernel of the reduction $\gl{2}{\oo_{f,\lambda}}\to\gl{2}{\overline{\mathbb{F}_\ell}}$
is torsion-free.

\begin{remark} For simplicity, we stated Theorem \ref{maintheorem} for classical newforms, but it also holds for Hilbert newforms of arithmetic weight. Indeed, their field of Hecke eigenvalues has finite degree over $\Q$, the associated Galois representations exist, and the analog of Carayol's lemma holds (see~\cite[Theorems 1.1 and~1.5]{Jar}), so all the necessary tools are there, and in fact, the same proof goes through.
\end{remark}

In the last section, we will use a level raising theorem due to Diamond and Taylor~\cite{DiaTay} to obtain examples of the application of Theorem~\ref{maintheorem}. In particular, we discuss a systematic method to generate newforms with primes that do not divide the level and are highly ramified in their field of coefficients.

\subsection*{Acknowledgements} We thank Luis Dieulefait for helpful conversations.

\section{Auxiliary results}\label{sec:preliminaires}
Let $f=\sum_{n=1}^\infty a_nq^n$ be a newform in $S_k(\Gamma_0(N),\varepsilon)$. We keep the notation of the introduction. The representation $\rho_{f,\lambda}$
is unramified outside $N\ell$, and for all primes $q\nmid N\ell$, the minimal polynomial of
$\rho_{f,\lambda}(\Frob_q)$ is $x^2-a_qx+q^{k-1}\varepsilon(q)$, where $\Frob_q \in G_\Q$ is a Frobenius element at~$q$. The determinant of $\rho_{f,\lambda}$ is equal to $\omega_\ell^{k-1}\varepsilon$, where $\omega_\ell:G_\Q\to \Z_\ell^*$ is the
$\ell$-adic cyclotomic character. Moreover, the set $\{ \rho_{f,\lambda}\}$ forms a compatible system. This means that, in particular, after choosing embeddings $G_{\Q_p}\hookrightarrow G_\Q$ at all primes~$p$, we have that, for $p \neq \ell$, the Weil--Deligne representation associated (via Grothendieck Monodromy Theorem) with~$\rho_{f,\lambda}|_{G_{\Q_p}}$
is independent of $\lambda$ and~$\ell$; see also \cite[Th\'eor\`eme A]{Car2} or \cite[Theorem 1.1]{Jar}.
Furthermore, Weil--Deligne representations fall into three types: Principal Series, Special (or Steinberg), and Supercuspidal (see e.g. \cite[\textsection 1]{DiePaTsa}), and those with infinite image of inertia are exactly the Special representations.

Recall that $N'$ is the prime-to-$\ell$ part of the Artin conductor of $\rsr{\rho_{f,\lambda}}$. When $v_p(N')<v_p(N)$, Carayol
characterized all possible Weil--Deligne types of $f$ at~$p$ (see~\cite[Proposition 2]{Car}) which we list in Proposition~\ref{Carayolcharacterization}, noting that there is a case
in {\it loc. cit} that does not occur under the assumption $p^2|N$ in Theorem~\ref{maintheorem}.
Let $E/\Q_\ell$ be a finite extension with ring of integers $\oo$.
For a representation $\rho:G_{\Q_p}\to\gl{n}{\oo}$
we denote by $a(\rho)$ the conductor exponent of $\rho$.
Let also $\overline{\rho}$ be the reduction of~$\rho$ modulo
the maximal ideal of~$\oo$ and $\rhobar^{ss}$ its semisimplification.

\begin{prop}\label{Carayolcharacterization}
Let $\rho: G_{\Q_p}\to \gl{2}{\oo}$ be as above.
Assume $a(\rho)\geq 2$ and
$$a(\rsr{\rho})<a(\rho).$$
Then $\rho$ satisfies one of the following cases.
    \begin{enumerate}
    \item $\rho$ is a principal series representation given by
    $$\rho = \chi_1\oplus \chi_2$$
    with $\overline{\chi_1}$ unramified. In this case, $a(\rho)=a(\chi_2)+1$ and $a(\rsr{\rho})=a(\overline{\chi_2}).$
    \item $\rho$ is a special representation given by
    $$\rho = \chi \otimes sp(2)$$
    with $a(\chi)=1$ and $a(\overline{\chi})=0$. In this case, we have $a(\rho)=2$ and $a(\rr{\rho})=0.$
    \item $\rho$ is a supercuspidal representation given by
    $$\rho = \text{Ind}_{G_{\Q _{p^2}}}^{G_{\Q_p}}\chi,$$
    where $\Q_{p^2}$ is the unique unramified quadratic extension of $\Q_p$ and the character $\chi$ does not factor through the norm $\Q_{p^2}\to\Q_p$ (where we view $\chi$ as a character of $\Q_{p^2}$ by local class field theory). In this case, $a(\rho)=2$, $a(\chi)=1$ and $a(\overline{\chi})=0.$
\end{enumerate}
\end{prop}

Finally, we will also need the following lemma (see e.g. \cite[Theorem 6.2]{SilZar}).
\begin{lemma}\label{torsionfreekrenel}
       Let $E/\Q_\ell$ be a finite extension with ring of integers $\oo$, ramification degree~$e$, and uniformizer $\pi$. The subgroup
       $I_n+\pi M_{n\times n}(\oo)$
    of $\gl{n}{\oo}$ is pro-$\ell$ group and, moreover, if $e<\ell-1$ then it is torsion-free.
\end{lemma}

\section{Proof of Theorem~\ref{maintheorem}}\label{sec:mainProof}

Assume the hypothesis of Theorem~\ref{maintheorem}, and keep
all the previous notations. We set
$$\rho := \rho_{f,\lambda}|_{G_{\Q_p}}\quad \text{ and } \quad \overline{\rho}:= \overline{\rho_{f,\lambda}}|_{G_{\Q_p}}.$$
Let $I_p \subset G_{\Q_p}$ be the inertia subgroup.
Since the proof of Theorem \ref{maintheorem} proceeds by analysing various cases, to simplify the presentation, we will split it into lemmas.

\begin{lemma}\label{theorem a case sp}
        Suppose that we are in the case (2) of Proposition~\ref{Carayolcharacterization}. Then
        $\Q(\zeta_\ell)\subset \Q_f$.
    \end{lemma}
\begin{proof}
The fact that $\chi|_{I_p}$ is not trivial but has trivial reduction implies that the values of $\chi$ on~$I_p$ contain a non-trivial $\ell$-power root of unity $\zeta_{\ell^k}$. Since $\ell>2$ the same is true for $\det(\rho|_{I_p}) = \eps|_{I_p} = \chi^2|_{I_p}$. Since the values
of $\eps$ lie in $\Q_f$, we conclude that $\zeta_\ell\in \Q_f.$
\end{proof}
\begin{lemma}\label{theorem b}
    The case $(b)$ of Theorem~\ref{maintheorem} holds.
    \end{lemma}
    \begin{proof}
        The assumption $v_p(N)>2$ implies that we are in case~(1)
        of Proposition~\ref{Carayolcharacterization}. Moreover, the conductor
        formula $a(\rho) = a(\chi_1) + a(\chi_2)$ gives $a(\chi_1)=1$ and $a(\chi_2)\geq 2$.
        Thus, by local class field theory, we can view $\chi_i|_{I_p}$ for $i=1,2$ as characters of 
         $$\chi_i|_{I_p}:\big(\Z/p\Z\big)^*\times (\text{finite~p-group}),$$
        where $\chi_1$ is nontrivial on the first factor (so in particular $p$>2) and trivial on the second factor, and $\chi_2$ is nontrivial on the second factor. In particular, we can find $ x_0\in$ $I_p$ such that $\chi_1(x_0)=1$ and $\chi_2(x_0)=\zeta_p$.
        Now, since $a(\overline{\chi}_1)=0$, the argument in the proof of Lemma~\ref{theorem a case sp} applied to $\chi_1$, gives that we can also find $x_1\in I_p$ such that $\chi_1(x_1)=\zeta_\ell$.

        Now, for $0\leq n< p$, we set $\sigma_n:=x_1x_0^n\in I_p$.
        We have
        $$\tr \rho(\sigma_n) =\chi_1(x_1x_0^n)+\chi_2(x_1x_0^n)=\zeta_\ell + \chi_2(x_1)\zeta_p^n \in \Q_{f,\lambda},$$
        and summing over $n$ gives
        $$\sum_{n=0}^{p-1}\tr\rho(\sigma_n)=p\zeta_\ell \implies \zeta_\ell \in \Q_{f,\lambda}.$$
        Since $\rho_{f,\lambda}|_{I_p}$ is independent of $\lambda \nmid Np$,
        we conclude that
        $\zeta_\ell\in \Q_{f,\mathfrak{q}}$
        for all primes $\mathfrak{q} \nmid Np$ in $\Q_f$. This means that the extension $\Q_f(\zeta_\ell)/\Q_f$ is of degree 1 because almost all primes are totally split in it, hence $\zeta_\ell\in\Q_f$.
    \end{proof}
    \begin{lemma}\label{case a ramification preserves}
    Assume that we are in case (1) or (3) of Proposition~\ref{Carayolcharacterization}.
    Suppose further that $v_p(N)=2$ and $v_p(N')=1$.
    Then $\Q(\zeta_\ell+\zeta_\ell^{-1})\subset \Q_f$.
    \end{lemma}
\begin{proof}
    Assume first we are in case (1) of Proposition~\ref{Carayolcharacterization}. We have $$a(\chi_1)=a(\chi_2)=a(\overline{\chi_2})=1,\ \ a(\overline{\chi_1})=0,$$
    By local class field theory, both characters factor through $(\Z/p\Z)^*$. Furthermore, we can find $x_0, x_1 \in I_p$ such that
    $$\chi_1(x_0)=1,\ \ \chi_2(x_0)=\zeta_q,\ \ \chi_1(x_1)=\zeta_\ell$$
    for some prime $q\neq \ell$ dividing $p-1$. Now, setting $\sigma_n = x_1 x_0^n$ for $0 \leq n < q$ and proceeding as in the proof of Lemma~\ref{theorem b},
    gives $\Q(\zeta_\ell) \subset \Q_f$.

    Assume now that we are in case (3) of Proposition~\ref{Carayolcharacterization}.
    We have $\rho|_{G_{\Q_{p^2}}} \simeq \chi \oplus \chi^s$, where $\chi^s$ denotes the conjugated character of~$\chi$ with respect to $\Q_{p^2}/\Q_p$. Also, the inertia subgroup inside $G_{\Q_{p^2}} \subset G_{\Q_p}$ is $I_p$.
    Hence, since $a(\chi)=1$ and $a(\overline{\chi})=0$, there is~$\sigma\in I_p$ such that $\chi(\sigma)=\zeta_\ell$. We have
    \[
     \rho(\sigma) = \chi(\sigma) \oplus \chi^s(\sigma), \quad \text{ and } \quad \chi(\sigma)\chi^s(\sigma) = \eps(\sigma).
    \]
    If $1 \neq \eps(\sigma) \in \Q_f$, then $\eps(\sigma)$ has order multiple of~$\ell$ and we get directly $\Q(\zeta_\ell) \subset \Q_f$.
    If $\eps(\sigma)=1$, then $\tr \rho(\sigma) =\zeta_\ell+\zeta_\ell^{-1} \in \Q_{f,\lambda}$ and arguing as at the end of the proof of Lemma~\ref{theorem b} yields $\Q(\zeta_\ell + \zeta_\ell^{-1})\subset \Q_f$.

\end{proof}
    \begin{remark}
        In the case of Hilbert modular forms over a totally real field $F$, the analog of the group $(\Z/p\Z)^*$ in Lemma \ref{theorem b} and Lemma \ref{case a ramification preserves} would be $\left(\Z/N_{F/\Q}(\mathfrak{p})\Z\right)^*$, where $\mathfrak{p}$ is a prime of $F$ dividing the level, and the rest of the proof stays the same.
    \end{remark}
\begin{lemma}\label{theorem a case ps and c}
    Assume that $\overline{\rho_{f,\lambda}}$ is irreducible and that we are in case (1) or (3) of Proposition~\ref{Carayolcharacterization}.
    Assume also $v_p(N)=2$ and $v_p(N')=0$.
    Then $\Q(\zeta_\ell+\zeta_\ell^{-1})\subset \Q_f$ and $e(\lambda|\ell)\geq \ell-1$.
    \end{lemma}
\begin{proof}
    Let $H$ be the kernel of reduction $\GL(\oo_{f,\lambda}) \to \GL(\overline{\F_\ell})$. The assumptions imply that the image of the inertia $\rho(I_p)$ is a finite, nontrivial subgroup of $H$. The inequality $e(\lambda|\ell)\geq \ell-1$ now follows immediately from Lemma \ref{torsionfreekrenel}. Since $H$ is a pro-$\ell$ group, we can find $\sigma\in I_p$
such that $\rho(\sigma)$ has order $\ell$. 
Therefore, either 
$$\det(\rho(\sigma)) =\zeta_\ell^n\neq 1 \quad \text{ or } \quad \tr \rho(\sigma) =\zeta_\ell+\zeta_{\ell}^{-1}.$$
Either way, we conclude that $\zeta_\ell+\zeta_{\ell}^{-1}\in~\Q_{f,\lambda}$. 
Since $\rho_{f,\lambda}|_{I_p}$ is independent of the prime $\lambda \nmid Np$, the same applies to the determinant and trace of $\rho_{f,\lambda}(\sigma).$ We conclude that $\zeta_\ell+\zeta_\ell^{-1}\in \Q_{f,\mathfrak{q}}$ for all primes $\mathfrak{q} \nmid Np$ in $\Q_f$. The conclusion follows as in the proof of Lemma~\ref{theorem b}.
\end{proof}

\begin{proof}[Proof of Theorem~\ref{maintheorem}] The first statement in part (a) follows from Lemmas~\ref{theorem a case sp},~\ref{case a ramification preserves} and~\ref{theorem a case ps and c}, and the last statement from Lemma~\ref{theorem a case ps and c}.
Part (b) follows from Lemma~\ref{theorem b}.
\end{proof}

\section{Examples of congruences and ramified primes}\label{sec:examples}

In this section, we give various examples related to our main theorem. 

We say that two newforms $f$ and $g$ are {\it congruent}, if there is a prime $\lambda''$ in the compositun $\Q_f\cdot \Q_g$ such that $\rsr{\rho_{f,\lambda}} \simeq \rsr{\rho_{g,\lambda'}}$, where $\lambda$ and $\lambda'$ are primes of, respectively, $\Q_f$ and $\Q_g$ induced by $\lambda''$. 
Moreover, such an isomorphism is equivalent to 
$$a_q(f) \equiv a_q(g) \pmod {\lambda''} \quad \text{ for all primes } q\nmid \ell p N',$$
In the special case $\Q_g = \Q$, 
we have $\lambda'' = \lambda$ and we say that 
$f$ and $g$ are {\it congruent modulo $\lambda$.}

\begin{example}\label{ex1}
Let $\ell=11$, $p=23$ and $N = 7406=23^2\cdot 14$. 
The newform $f\in S_2(\Gamma_0(N))$ 
with LMFDB~\cite{lmfdb} label \LMFDBN{7406.2.a.bt}  
satisfies $[\Q_f : \Q] = 20$ and $\Q_f$ contains $\Q(\zeta_{11}+\zeta_{11}^{-1})$. Moreover, there are three primes $\lambda,\lambda_1,\lambda_2$ in $\Q_f$ above~$11$ satisfying 
$$e(\lambda|11)=10,\ \ e(\lambda_1|11)=e(\lambda_2|11)=5.$$
Let $E$ be the elliptic curve with Cremona label \LMFDBE{14a1}.
There is a congruence modulo~$\lambda$ between~$f$ and the newform 
$g \in S_2(\Gamma_0(14))$ 
corresponding to the isogeny class of~$E$, 
that is,
$$a_q(f) \equiv a_q(E) \pmod{\lambda}$$
for all primes $q\nmid 11\cdot7406$.  
A quick consultation of LMFDB shows that $\rr{\rho_{E,11}}$ is irreducible. Hence $\rr{\rho_{f,\lambda}}$ is irreducible as well, and we have
$\rr{\rho_{f,\lambda}}\simeq\rr{\rho_{E,11}}$.
Since $E$ has multiplicative reduction at 2 and 7 and minimal discriminant $\Delta(E) = 2^9 \cdot 7^2$, by the theory of the Tate curve, we know that $\rr{\rho_{E,11}}$ ramifies at 2 and 7. Thus $N' = 14$.

Hence, $f$ satisfies the hypothesis of part $a)$ 
of Theorem~\ref{maintheorem} including the "moreover part", therefore $\Q(\zeta_{11}+\zeta_{11}^{-1}) \subset \Q_f$ and $e(\lambda|11) \geq 10$ which is compatible with the observations above. 
\end{example}

The following example is the only case of ramification loss not covered by Theorem~\ref{maintheorem}. In particular, it shows that the assumption $p^2 \mid N$ is necessary to guarantee ramification in the field of coefficients.

\begin{example}
Let $\ell=7$, $p=5$ and $N=55$. 
The newform $f\in S_2(\Gamma_0(N))$ with LMFDB label \LMFDBN{55.2.a.b}
has field of coefficients $\Q_f = \Q(\sqrt{2})$. The prime 
$\lambda=(3+\sqrt{2})$ divides $7$ in $\Q_f$ and there is a congruence modulo $\lambda$ between $f$ and the unique newform $g\in S_2(\Gamma_0(11))$ which corresponds to the isogeny class of an elliptic curve of conductor 11. In this case, we have 
$$5 \Vert N,\qquad  5\nmid N',$$
and clearly there is no ramification at $7$ in $\Q_f$. 
\end{example}
To finish, we will describe a systematic computation to generate newforms with primes (very) ramified in their field of coefficients based on 
part a) of Theorem~\ref{maintheorem} and Example~\ref{ex1}. 

Indeed, to recreate the situation $v_p(N)=2$ and $v_p(N')=0$ (Table \ref{Table 2,0}) or $v_p(N')=1$ (Table \ref{Table 2,1}) we will use the following proposition, which is a special case of \cite[Theorem A]{DiaTay}.

\begin{prop}\label{DiamondTaylor}
    Let $\lambda$ be a prime of $\overline{\Q}$ over an integer prime $\ell>3$. Let $g\in S_2(\Gamma_0(N'),\eps')$ be a newform of level $N'$, weight $2$, and nebentypus $\eps'$ with $\ell\nmid N'$, such that the residual representation
    $$\overline{\rho_{g,\lambda}}:G_\Q\to\gl{2}{\overline{\mathbb{F}_\ell}}$$
    is irreducible and has Artin conductor $N'$. Let $\alpha\in\{1,2\}.$ If one of the following is satisfied
    \begin{enumerate}[1)]
        \item $\ell\mid p+1$ and
        \begin{enumerate}
            \item $p\nmid N'$, $\ell \mid a_p(g)$, and $\alpha=2$, or
            \item $p\Vert N' $, $\eps'$ is unramified at $p$, and $\alpha=1$;
        \end{enumerate}
        \item $\ell \mid p-1$ and
        \begin{enumerate}
            \item $p\nmid N'$ and $\alpha=2$, or
            \item $v_p(\mathfrak{f}(\eps'))=v_p(N')\geq 2$ and $\alpha=1$, where $\mathfrak{f}(\eps')$ is the conductor of $\eps'.$
        \end{enumerate}
    \end{enumerate}
    then there exists a nebentypus $\varepsilon$ and a newform $f\in S_2(\Gamma_0(p^\alpha N'),\varepsilon)$ such that
    $$\overline{\rho_{f,\lambda}}\simeq \overline{\rho_{g,\lambda}}.$$
\end{prop}
Using the previous proposition, the idea is as follows.
\begin{enumerate}[1.]
    \item Pick a modular form $g$ of small level $N'$ and weight $2$ and trivial nebentypus.
    \item Choose a pair of primes $(\ell,p)$ satisfying the conditions of Proposition \ref{DiamondTaylor}.
    \item Check that the representation $\rr{\rho_{g,\lambda}}$ is irreducible and has conductor $N'$.
    \item Find $f$ in $S_2(\Gamma_1(p^2N'))$ congruent to $g$.
    \item Calculate $\Q_f$ and the splitting of $\ell$ in $\Q_f$.
\end{enumerate}
In Tables \ref{Table 2,0} and \ref{Table 2,1} we gather various congruences together with the splitting behaviour of $\ell$ in~$\Q_f.$ The most important column for us is the fifth one, which shows the ramification degrees of primes in $\Q_f$ above $\ell$; the last column shows their inertial degrees. We see in particular that for all of them the ramification degree is divisible by $(\ell-1)/2$. Moreover, there is always one prime $\lambda$ such that $e(\lambda|\ell)\geq \ell-1$, which is the prime of the congruence.
\begin{remark}
    For the tables, we used the LMFDB label to describe modular forms. In both tables, all newforms~$g$, except \LMFDBN{74.2.a.a}, correspond to isogeny classes of elliptic curves, so $\Q_g = \Q \subset \Q_f$ in all such cases. In the exceptional case of Table \ref{Table 2,1}, $\mathbb{Q}_g=\Q(\sqrt{13})$ and $p=19$ stays prime in $\mathbb{Q}_g.$ In this case, to find the congruence between $f$ and $g$, 
    it is enough to compare the two reductions of $g$ modulo (19) (related by the nontrivial automorphism of~$\F_{19^2}$) to the reductions of a fixed representative of the Galois orbit of~$f$ modulo the primes $\lambda \mid \ell$ in $\Q_f$ (this is a comparison in $\overline{\mathbb{F}_{19}}$). Indeed, if $g$ is congruent to $f^\sigma$ modulo~$\lambda$ for some $\sigma \in G_\Q$ and $\lambda \mid \ell$ i n $\Q_f \cdot \Q_g$, then~$g^{\sigma^{-1}}$ is congruent to $f$ modulo $\lambda^{\sigma^{-1}}$. Since $19$ is inert in~$\Q_g$, conjugating $g$ commutes with taking reductions.
\end{remark}
\begin{table}[!ht]
\begin{tabular}{|l|l|l|l|l|l|}
\hline
$(\ell,p,N')$ & $g$ & $f$ & Degree of $\mathbb{Q}_f$ & ram. degrees & in. degrees  \\ \hline
$(11,23,14)$ & \LMFDBN{14.2.a.a} & \LMFDBN{7406.2.a.bt} & 20  & 10,5,5 & 1,1,1 \\ \hline
$(11,23,15)$ & \LMFDBN{15.2.a.a} & \LMFDBN{7935.2.a.bu} & 25 & 10,5,5,5 & 1,1,1,1 \\ \hline
$(11,23,20)$ & \LMFDBN{20.2.a.a} & \LMFDBN{10580.2.a.bg} & 25 & 10,5,5,5  & 1,1,1,1\\ \hline
$(11,23,21)$ & \LMFDBN{21.2.a.a} & \LMFDBN{11109.2.a.ch} & 40 & 10,5,5,5 & 1,1,2,3\\ \hline
$(11,23,24)$ & \LMFDBN{24.2.a.a} & \LMFDBN{12696.2.a.bt} & 15 & 10,5 & 1,1\\ \hline
$(11,23,26)$ & \LMFDBN{26.2.a.a} & \LMFDBN{13754.2.a.ci} & 30 & 10,5,5,5 & 1,1,1,2\\ \hline
$(11,23,26)$ & \LMFDBN{26.2.a.b} & \LMFDBN{13754.2.a.cm} & 30 & 10,5,5,5 & 1,1,1,2\\ \hline
$(7,29,11)$ & \LMFDBN{11.2.a.a} & \LMFDBN{9251.2.a.z} & 39 & 6,3,3,3 & 1,1,3,7\\ \hline
$(7,29,15)$ & \LMFDBN{15.2.a.a} & \LMFDBN{12615.2.a.bv} & 15 & 12,3 & 1,1\\ \hline
$(7,29,20)$ & \LMFDBN{20.2.a.a} & \LMFDBN{16820.2.a.v} & 18 & 6,3,3 & 1,1,3\\ \hline
$(7,13,88)$ & \LMFDBN{88.2.a.a} & \LMFDBN{14872.2.a.bn} & 18 & 6,3,3 & 1,2,2\\ \hline
${(7,13,114)}$ & \LMFDBN{114.2.a.a} & \LMFDBN{19266.2.a.cs} & 6 & 6 & 1\\ \hline
$(5,19,39)$ & \LMFDBN{39.2.a.a} & \LMFDBN{14079.2.a.bo} & 14 & 4,2,2 & 1,2,3\\ \hline
${(5,11,14)}$ & \LMFDBN{14.2.a.a} & \LMFDBN{1694.2.a.x} & 4 & 4 & 1 \\ \hline
$(5,11,17)$ & \LMFDBN{17.2.a.a} & \LMFDBN{2057.2.a.bd} & 18 & 4,4,2,2 & 1,1,2,3\\ \hline
\end{tabular}
\caption{Case $v_p(N)=2,\ v_p(N')=0$. We have $\rr{\rho_{f,\lambda}} \simeq \rr{\rho_{g,\ell}}$, where $\lambda \mid \ell$ is the prime with highest ramification degree in $\Q_f$; in the last row, where there are two such primes, the congruence exists modulo only one of them.}
\label{Table 2,0}
\end{table}
\begin{table}[!ht]
\begin{tabular}{|l|l|l|l|l|l|}
\hline
($\ell,p,N')$  & $g$ & $f$ & Degree of $\mathbb{Q}_f$ & ram. degrees & in. degrees \\ \hline
($5,19,19)$ &  \LMFDBN{19.2.a.a} &  \LMFDBN{361.2.a.d} & 2 & 2 & 1\\ \hline 
($7,13,26)$ &  \LMFDBN{26.2.a.a} &  \LMFDBN{338.2.a.g} & 3 & 3 & 1\\ \hline 
($7,13,39)$ &  \LMFDBN{39.2.a.a} &  \LMFDBN{507.2.a.l} & 3 & 3 & 1\\ \hline 
($7,41,82)$ &  \LMFDBN{82.2.a.a} &  \LMFDBN{3362.2.a.r} & 3 & 3 & 1\\ \hline
($11,43,43)$ &  \LMFDBN{43.2.a.a} &  \LMFDBN{1849.2.a.r} & 20 & 5,5,5 & 1,1,2\\ \hline 
($17,67,67)$ &  \LMFDBN{67.2.a.a} &  \LMFDBN{4489.2.a.o} & 48 & 8,8,8,8 & 1,1,2,2\\ \hline  
($19,37,37)$ &  \LMFDBN{37.2.a.a} &  \LMFDBN{1369.2.a.o} & 27 & 9,9,9 & 1,1,1\\ \hline 
($19,37,37)$ &  \LMFDBN{37.2.a.b} &  \LMFDBN{1369.2.a.o} & 27 & 9,9,9 & 1,1,1\\ \hline 
($19,37,74)$ &  \LMFDBN{74.2.a.a} &  \LMFDBN{2738.2.a.w} & 18 & 9 & 2\\ \hline 
\end{tabular}
\caption{Case $v_p(N)=2,\ v_p(N')=1$.}
\label{Table 2,1}
\end{table}

Let us now examine the algorithm closely. To check the irreducibility of $\rr{\rho_{g,\lambda}}$ it suffices to find an irreducibile characteristic polynomial of $\rr{\rho_{g,\lambda}}(\Frob_q)$ for a prime $q\nmid \ell pN'$ (because for odd representation of $G_\Q$ being irreducible and absolute irreducible are equivalent). To check that the conductor of the representation is indeed $N'$, by the famous theorem of Khare and Wintenberger, it is enough to check that $g$ is not congruent to any other modular form of weight 2 and with level dividing $N'$ (since $N'$ is small, there will be only a few of such forms if any).
Step 4 is the hardest one because, in general, calculating the space $S_2(\Gamma_0(N),\varepsilon)$ with nontrivial $\varepsilon$ is hard. Notice, however, that for all $f$ in the table, the letter following~"2." is~"a", which means that $f$ has trivial nebentypus. This is not a coincidence; in fact, the following lemma reduces the search to forms with trivial nebentypus.

\begin{lemma} Let $g\in S_2(\Gamma_0(N'))$ be a newform. 
Let $\alpha \in \{1,2\}$ and $\lambda,\ell,p$ be primes such that $g$ satisfies one of the conditions of  Proposition~\ref{DiamondTaylor}. Assume also $\ell\nmid \varphi(N')$, where~$\varphi$ is Euler totient function. Then we can find $f$ congruent to $g$ modulo~$\lambda$ in $S_2(\Gamma_0(p^\alpha N')).$
\end{lemma}
\begin{proof} Observe that $p > 3$ because $\ell \mid (p-1)(p+1)$ and $\ell > 3$.
    By Proposition~\ref{DiamondTaylor}, there exists a newform $f'\in S_2(\Gamma_0(p^\alpha N'),\varepsilon)$ congruent to $g$ modulo~$\lambda$. Assume that $\varepsilon$ is nontrivial; otherwise, there is nothing to show. Comparing determinants, we must have $\overline{\varepsilon}=1$ which means that the order of $\varepsilon$ is an $\ell$-power, say $\ell^\beta$. 
    Moreover, $\varepsilon$ being a Dirichlet character of modulo $p^\alpha N'$, we have $ \ell^\beta \mid \varphi(p^\alpha N')$.

    If $p\mid N'$, then $\varphi(p^\alpha N')=p^\alpha\varphi(N')$ is not divisible by $\ell$ by assumption, a contradiction. Therefore, $p\nmid N'$ and $\varphi(p^\alpha N')=(p-1)p^{\alpha-1}\varphi(N')$. 
    We conclude that $\ell^\beta \mid p-1$ and $\varepsilon$ factors via $(\Z/p\Z)^*$ and so it
    has conductor~$p$. Furthermore, we are in case 2) (a) of Proposition~\ref{DiamondTaylor} and, in particular, $\alpha = 2$.

    Denote by $\rho'$ the representation $\rho_{f',\lambda}$. By assumption, we have
    $$\overline{\rho_{f',\lambda}}\simeq \overline{\rho_{g,\lambda}}.$$
    Since $p\nmid N'$, the inertia type of $g$ at $p$ is unramified principal series, i.e,
    $$\rho_{g,\lambda}|_{G_{\Q_p}}\simeq \psi_1\oplus \psi_2$$
    for some unramified characters $\psi_1,\psi_2.$
    Consider the following data:
    \begin{itemize}
        \item the restriction $\rho_r:=\rho'|_{G_{\Q_r}}$ for every prime $r\mid \ell N',$
        \item the Principal Series representation $\rho_p:=\psi_1\eps\oplus \psi_2\eps^{-1}$ of $G_{\Q_p},$
        \item the $\ell$-adic cyclotomic character $\psi=\omega_\ell.$
    \end{itemize}
    From \cite[Lemma 1.13, Lemma 1.14]{DiPa} the representation $\overline{\rho_{g,\lambda}}|_{G_{\Q(\zeta_l)}}$ is absolutely irreducible. This allows us to apply \cite[Th\'eor\`eme 3.2.2]{BrDi} to get a newform $f$ such that $\rho_{f,\lambda}$ is unramified outside $p\ell N'$ with determinant $\psi$, and moreover,
    $$\rho_{f,\lambda}|_{G_{\Q_r}}\simeq \rho_r,\ \ \text{for every }r\mid \ell p N',$$
    $$\overline{\rho_{f,\lambda}}\simeq\overline{\rho_{g,\lambda}}.$$
    The local conditions, together with the determinant, imply that $f$ is a newform in $S_2(\Gamma_0(p^2 N'))$.
\end{proof}
 Showing that two newforms are congruent can still be hard (if Sturm's bound is big). In our case, we got an advantage because we already know that there is some newform congruent to~$g$, so following the famous logic of Sherlock Holmes, instead of showing that a certain form $f$ is congruent to $g$, it is enough to show that all the other newforms are not congruent to $g$, which is a lot easier.

\begin{example}
From Table \ref{Table 2,0} we have another method for obtaining examples of the case $v_p(N')=2$ and $v_p(N')=1.$ Indeed, as a consequence of Carayol's Theorem (Proposition~\ref{Carayolcharacterization}), reducing any $f$ from Table \ref{Table 2,0} modulo any prime $\lambda'$ of $\Q_f$ kills some ramification modulo~$p$. If $e(\lambda|\ell)<\ell-1$ and $\overline{\rho_{f,\lambda'}}$ is irreducible, then by Theorem \ref{maintheorem} the only possiblity is 
$$v_p(N)=2\ \ \text{and}\ \ v_p(N')=1.$$    
\end{example}

\begin{remark}
Finally, let us comment on the case b) of Theorem \ref{maintheorem}. The case $v_p(N)>2$ leads rapidly to very high levels and very big coefficient fields. Indeed, to apply the method above, the level of $f$ would have to be divisible by $p^3$, where $\ell\mid p-1$ for some $\ell>3$, and we are in case 2)(b) of Proposition~\ref{DiamondTaylor}. 
The third smallest such $p$ is $p=29$ (here $\ell=7$), leading to $29^3=24389\mid N$. Moreover, the condition $v_p(\mathfrak{f}(\eps'))=v_p(N')\geq 2$ forces choosing $g$ with non-trivial nebentypus, leading to harder computations and higher degree coefficient fields. For example, already for the two smallest cases
$(\ell,p) = (5,11),(11,23)$ the possibilities for $g$ have coefficient fields of degrees $\geq 100$.
\end{remark}

\bibliographystyle{amsplain}
\bibliography{ref}
\end{document}